\documentclass[11pt,a4paper]{amsart}
\usepackage{kvsetkeys}
\usepackage[utf8]{inputenc}
\usepackage{amssymb}
\usepackage{amsmath}
\usepackage{graphicx}
\usepackage{amsfonts}
\usepackage{mathrsfs}
\usepackage{yhmath}
\usepackage[margin=1in]{geometry}
\newcommand{\ol }{\overline }

\newcommand{\abs}[1]{\left\lvert #1 \right\rvert}

\newcommand{\ZZ }{\mathbb{Z}}
\newcommand{\QQ }{\mathbb{Q}}
\newcommand{\RR }{\mathbb{R}}

\newcommand{\eps }{\varepsilon }

\newcommand{\grond }{\mathcal{G}}

\newcommand{\injto }{\hookrightarrow }

\PassOptionsToPackage{usenames,svgnames,dvipsnames,table}{xcolor}	
\usepackage{xcolor}
\usepackage[colorlinks=true]{hyperref}
\hypersetup{urlcolor=RubineRed,linkcolor=RoyalBlue,citecolor=ForestGreen}
\usepackage[nameinlink]{cleveref}
\usepackage{amsthm}
\usepackage{thmtools}
\usepackage[framemethod=TikZ]{mdframed}
\mdfdefinestyle{mdbluebox}{
roundcorner = 10pt,
linewidth = 1pt,
skipabove = 12pt,
innerbottommargin = 9pt,
skipbelow = 2pt,
linecolor = blue,
nobreak = true,
backgroundcolor = TealBlue!5,
}
\declaretheoremstyle[
headfont=\sffamily\bfseries\color{MidnightBlue},
mdframed={style=mdbluebox},
headpunct={\\[3pt]},
postheadspace={0pt}
]{thmbluebox}
\mdfdefinestyle{mdredbox}{
linewidth=0.5pt,
skipabove=12pt,
frametitleaboveskip=5pt,
frametitlebelowskip=0pt,
skipbelow=2pt,
frametitlefont=\bfseries,
innertopmargin=4pt,
innerbottommargin=8pt,
nobreak=true,
backgroundcolor=Salmon!5,
linecolor=RawSienna,
}
\declaretheoremstyle[
headfont=\bfseries\color{RawSienna},
mdframed={style=mdredbox},
headpunct={\\[3pt]},
postheadspace={0pt},
]{thmredbox}
\mdfdefinestyle{mdgreenbox}{
skipabove=8pt,
linewidth=2pt,
rightline=false,
leftline=true,
topline=false,
bottomline=false,
linecolor=ForestGreen,
backgroundcolor=ForestGreen!5,
}
\declaretheoremstyle[
headfont=\bfseries\sffamily\color{ForestGreen!70!black},
bodyfont=\normalfont,
spaceabove=2pt,
spacebelow=1pt,
mdframed={style=mdgreenbox},
headpunct={ --- },
]{thmgreenbox}
\mdfdefinestyle{mdblackbox}{
skipabove=8pt,
linewidth=3pt,
rightline=false,
leftline=true,
topline=false,
bottomline=false,
linecolor=black,			
backgroundcolor=RedViolet!5!gray!5,
}
\declaretheoremstyle[
headfont=\bfseries,
bodyfont=\normalfont\small,
spaceabove=0pt,
spacebelow=0pt,
mdframed={style=mdblackbox}
]{thmblackbox}
\theoremstyle{definition}
\declaretheorem[name=Theorem,numberwithin=section]{theorem}

\declaretheorem[name=Proposition,sibling=theorem]{proposition}

\declaretheorem[name=Theorem,numbered=no]{theorem*}
\declaretheorem[name=Lemma,numbered=no]{lemma*}
\declaretheorem[name=Proposition,numbered=no]{proposition*}
\declaretheorem[name=Corollary,numbered=no]{corollary*}

\declaretheorem[name=Definition,numbered=no]{definition*}

\declaretheorem[style=thmgreenbox,name=Algorithm,numbered=no]{algorithm*}

\declaretheorem[style=thmgreenbox,name=Claim,numbered=no]{claim*}

\declaretheorem[style=thmgreenbox,name=Case,numbered=no]{case*}

\declaretheorem[style=thmredbox,name=Example,numbered=no]{example*}

\declaretheorem[style=thmredbox,name=Examples,numbered=no]{examples*}

\declaretheorem[style=thmblackbox,name=Remark,numbered=no]{remark*}

\declaretheorem[style=thmblackbox,name=Notation,numbered=no]{notation*}

\declaretheorem[name=Conjecture,numbered=no]{conjecture*}

\declaretheorem[name=Exercise,numbered=no]{exercise*}

\declaretheorem[name=Fact,numbered=no]{fact*}

\declaretheorem[name=Problem,numbered=no]{problem*}

\declaretheorem[name=Question,numbered=no]{question*}

\declaretheorem[name=Abuse of notation,numbered=no]{abuse*}
\usepackage{fancyhdr}
\begin{document}
\title{A Generalisation of Euler's Totient Function}
\author{Vlad Robu}
\address{Faculty of Mathematics and Informatics\\
	University of Bucharest,
	Academiei st. 14,
	Bucharest, Romania}
	\email{vlad.robu@s.unibuc.ro}
\date{\today}
\maketitle

\begin{abstract}
 Euler's totient function, $\varphi(n)$, which counts how many of $0,1,\dots,n-1$ are coprime to $n$, has an explicit asymptotic lower bound of $n/\log \log n$, modulo some constant. In this note, we generalise $\varphi$; given an irreducible integer polynomial $P$, we define the arithmetic function $\varphi_P(n)$ that counts the amount of numbers among $P(0),P(1),\dots,P(n-1)$ that are coprime to $n$. We also provide an asymptotic lower bound for $\varphi_P(n)$. 





\end{abstract}

\section{Introduction and Preliminary Remarks}

Throughout the entire text, $p$ is intended to be a prime number. For functions $f,g:\RR\to \RR_{>0}$, we say that $f\sim g$ if $\displaystyle \lim_{x\to \infty}f(x)/g(x)=1$.

We start by recalling the following result, which is due to Mertens: 

\begin{theorem}
[Mertens' third theorem]
We have that $$ \prod_{\substack{p\le x}}\left(1-\frac{1}{p}\right)\sim  \frac{e^{-\gamma}}{\log x},$$ 
where $\gamma$ is the Euler-Mascheroni constant, $\gamma=\displaystyle \lim_{n\to\infty}\left(H_n-\log n\right)$.
\end{theorem}

Based on this, a quick \textit{lower bound} for $\varphi(n)$ can be established asymptotically, where $\varphi$ indicates the Euler totient function. Indeed, we know that \cite{landau}
$$\frac{\varphi(n)}{n}=\prod_{p\mid n}\left(1-\frac{1}{p}\right)=\prod_{\substack{p\mid n\\p\le \log n}}\left(1-\frac{1}{p}\right)\prod_{\substack{p\mid n\\p>\log n}}\left(1-\frac{1}{p}\right).$$
The second product cannot have more than $\log n/\log \log n$ factors. Thus it is bounded below by $\displaystyle \left(1-\frac{1}{\log n}\right)^{\log n/\log \log n}$, which goes to $1$ (from below) as $n$ goes to infinity. The first product is bounded below by $\displaystyle \prod_{p\le \log n}\left(1-\frac{1}{p}\right)\sim \frac{e^{-\gamma}}{\log \log n}$. Overall, we obtain that $\displaystyle \varphi(n)$ is bounded below by a quantity asymptotically similar to $\displaystyle e^{-\gamma}\cdot\frac{n}{\log \log n}$. In this note, we are going to generalise $\varphi$ and also provide a lower bound for the generalised version; however, by contrast to the case of the usual $\varphi$, the constants that appear in front will not be explicit. 

Mertens' third theorem can be rapidly generalised. Let us take a look at the asymptotics of the quantity $\displaystyle \prod_{d<p\le x}\left(1-\frac{d}{p}\right)$ in terms of $x$, where $d$ is a fixed positive integer. The whole trick is that it can be connected to the product in Mertens' third theorem via the \textit{convergent} product \cite{generalise.mertens}
$$\grond_d=\prod_{p>d}\left(1-\frac{d}{p}\right)\left(1-\frac{1}{p}\right)^{-d}.$$
Indeed, $\grond_d$ converges, because we can write $\displaystyle\log \grond_d=\sum_{p>d}\left(\log\left(1-\frac{d}{p}\right)-d\log \left(1-\frac{1}{p}\right)\right)$. As $\log(1-x)=-x+O(x^2)$, we get that $\displaystyle \log\left(1-\frac{d}{p}\right)-d\log \left(1-\frac{1}{p}\right)=O\left(\frac{1}{p^2}\right)$, so the fact that $\displaystyle \sum_{p}\frac{1}{p^2}<\infty$ finishes it. With this, 

\begin{align*}
    \prod_{d<p\le x}\left(1-\frac{d}{p}\right)&=\prod_{p\le d}\left(1-\frac{1}{p}\right)^{-d}\cdot \prod_{p\le x}\left(1-\frac{1}{p}\right)^d\cdot \prod_{d<p\le x}\left(1-\frac{d}{p}\right)\left(1-\frac{1}{p}\right)^{-d}\sim\\
    &\sim \prod_{p\le d}\left(1-\frac{1}{p}\right)^{-d}\cdot\frac{e^{-d\gamma}}{(\log x)^d}\cdot \grond_d.
\end{align*}

Let us continue with some notes on Galois theory. Consider an irreducible polynomial $f$ with integer coefficients, with splitting field $K$ and Galois group $G$; also, let $n=\deg f$ be the degree of $f$. We view $G$ in the usual manner, as a group of permutations on the roots of $f$, embedded in $S_n$. 

The following two results establish a strong connection between the cycle decomposition of permutations in $G$ and the way $f$ decomposes modulo primes $p$. 

\begin{theorem}
[Dedekind]
(see e.g. \cite{reducible mod all primes} or theorem 8.23 from \cite{milne})
Consider a prime $p$ that does not divide the discriminant of $K$. If $f$ decomposes in $\ZZ_p[X]$ as $\displaystyle f=\prod_{i=1}^tf_i$, then $G$ contains at least one permutation which is the product of $t$ disjoint cycles of lengths $n_1,\dots,n_t$, where $n_i$ is the degree of $f_i$.
\end{theorem}

\begin{theorem}
[Frobenius density theorem]
(see e.g. \cite{reducible mod all primes} or \cite{lenstra&stevenhagen}, p. 11)
If $G$ contains a permutation $\sigma$ which is the product of disjoint cycles of length $n_1,\dots,n_t$, then there exists an infinite set $P_\sigma$ of primes such that for any $p\in P_\sigma$, the decomposition of $f$ modulo $p$ has the following structure: $\displaystyle f=\prod_{i=1}^tf_i$, where $f_i$ is a polynomial in $\ZZ_p[X]$ of degree $n_i$, for all $i\in \{1,2,\dots,t\}$.

Moreover, the set $P_\sigma$ has natural density equal to $1/\abs{G}$ times the number of permutations $\sigma$ of $G$ with cycle decomposition type $(n_1,\dots,n_t)$. 
\end{theorem}

We will be especially interested in the latter, as it will allow us to find the exact asymptotics of some products along the way. 

Before finally moving to the main stage, let us recall the following
\begin{proposition}\label{isotropy}
$G$ acts transitively on the set of roots of $f$, since $f$ is irreducible. Therefore, the number of automorphisms of $G$ that fix a given root of $f$ is the same for all roots of $f$. 
\end{proposition}

\section{Main Result}

\begin{theorem}
Consider a nonconstant irreducible polynomial $P\in \ZZ[X]$ of degree $d=\deg P$ and define $\delta=\displaystyle\gcd_{k\ge 0}P(k)$. Let $G\injto S_d$ be the Galois group of $P$ over $\QQ$ and also let $q$ be the proportion of permutations in $G$ that fix a given root of $P$. Then there exists a constant $c>0$, such that for all $\eps>0$ and all $n\in \ZZ_{\ge 2}$ with $\gcd(n,\delta)=1$ that are greater than some $n(\eps)$, among
$$P(0),P(1),P(2),\dots,P(n-1)$$
there are at least $cn\cdot\left(\log \log n\right)^{-qd-\eps}$ numbers relatively prime to $n$.
\end{theorem}

\begin{proof}
[Proof of Theorem 2.1]
Denote by $\varphi _P(n)$ as the number of integers among $P(0)$, $P(1)$, $P(2)$, $\dots$, $P(n-1)$ which are relatively prime to $n$, for all $n\in \ZZ_{\ge 2}$. Also, define by $f(n)$ the number of multiples of $n$ among $P(0)$, $P(1)$, $P(2)$, $\dots$, $P(n-1)$, for all $n\in \ZZ_{\ge 2}$. We have the following
\begin{lemma*}
With the above mentioned notations, 
$$\varphi _P(n)=
n\prod _{\substack{
p|n
}}
\left (1-\frac{f(p)}{p}\right ).
$$
\end{lemma*}

For a proof of the lemma, see e.g. \cite{NTconcepts&problems}, pp. 404-405. So we have to work on the product 
$$\frac{\varphi_P(n)}{n}=\prod_{\substack{p\mid n}}\left(1-\frac{f(p)}{p}\right).$$

Let $d=\deg P$ and consider from now on only positive integers $n>e^{d+1}$. Also, fix some small $\eps>0$. Then we can split our product into $3$ parts:
\begin{align*}
    \frac{\varphi_P(n)}{n}&=\prod_{\substack{p\mid n\\p\le d}}\left(1-\frac{f(p)}{p}\right)\cdot \prod_{\substack{p\mid n\\d+1\le p<\log n}}\left(1-\frac{f(p)}{p}\right)\cdot\prod_{\substack{p\mid n\\p\ge \log n}}\left(1-\frac{f(p)}{p}\right)=\\
    &=\Pi_1\Pi_2\Pi_3.
\end{align*}

For $p\le d$ that divides $n$, since $\gcd(\delta,n)=1$, it follows that not all of the numbers $P(i)$, $i=\ol{0,p-1}$, are divisible by $p$, hence $f(p)\le p-1$. Consequently, $\displaystyle \Pi_1\ge \prod_{p\le d}\frac{1}{p}$.

If $p\ge d+1$, then the equation $P(k)=0$ cannot have more than $\deg P=d$ solutions modulo $p$. Consequently, $f(p)\le d$. 

In the prime decomposition of $n$, if $k$ is the number of distinct prime factors exceeding $\log n$, then $n\ge (\log n)^k$, so $k\le \displaystyle \frac{\log n}{\log \log n}$. It follows that $\Pi_3\ge \displaystyle \prod_{\substack{p\mid n\\p\ge \log n}}\left(1-\frac{d}{p}\right)\ge \left(1-\frac{d}{\log n}\right)^{\frac{\log n}{\log \log n}}\sim e^{-\frac{d}{\log \log n}}\ge 0.99$ for all sufficiently large positive integers $n$. 

For $\Pi_2$, we need to be a bit more careful when dealing with the $f(p)$'s. If we define $g(p)$ the number of roots $P$ has modulo $p$, we obviously have $f(p)\le g(p)$, as $f(p)$ counts only the number of distinct roots modulo $p$. For each $k\in\{0,1,\dots,d\}$, define $P_k=g^{-1}(k)$ to be the set of those primes $p>d$ such that $g(p)=k$. Then, by writing $x=\log n$, we more generally look at the quantity
$$\Pi_2\ge \Pi_2'=\prod_{k=0}^d\prod_{\substack{d<p\le x\\p\in P_k}}\left(1-\frac{k}{p}\right).$$

Now, if we define $\displaystyle \grond_{d,k}=\prod_{\substack{p>d\\p\in P_k}}\left(1-\frac{k}{p}\right)\left(1-\frac{1}{p}\right)^{-k}$, in a similar manner to establishing the convergence of $\grond_d$, we get that $\grond_{d,k}$ is convergent too, for all $k\in\{0,1,\dots,d\}$. Consequently,
\begin{align*}
    \Pi_2'&=\prod_{k=0}^d\prod_{\substack{d<p\le x\\p\in P_k}}\left(1-\frac{k}{p}\right)=\\
    &=\prod_{k=0}^d\left(\prod_{\substack{d<p\le x\\p\in P_k}}\left(1-\frac{k}{p}\right)\left(1-\frac{1}{p}\right)^{-k}\right)\cdot \prod_{k=0}^d\left(\prod_{\substack{d<p\le x\\p\in P_k}}\left(1-\frac{1}{p}\right)\right)^k\sim \\
    &\sim \prod_{k=0}^d\left(\grond_{d,k}\cdot\prod_{p\le d}\left(1-\frac{1}{p}\right)^{-k}\right)\cdot \prod_{k=0}^d\left(\prod_{\substack{p\le x\\p\in P_k}}\left(1-\frac 1p\right)\right)^k.
\end{align*}

Therefore, we are interested in the asymptotics of $\displaystyle \prod_{\substack{p\le x\\p\in P_k}}\left(1-\frac 1p\right)$. For $p\in P_k$, $g(p)=k$ means that the polynomial $P$ has exactly $k$ roots modulo $p$, counted with multiplicity. By Frobenius' Density Theorem, we know that $P_k$ has natural density $\alpha_k=\beta_k/\abs{G}$, where $\beta_k$ counts the number of permutations in $G$ that have exactly $k$ fixed points. This implies that 
$$\lim_{x\to \infty}\frac{\sum_{\substack{p\in P_k,~p\le x}}1/p}{\sum_{p\le x}1/p}=\alpha_k.$$
So, for all sufficiently large $x$, we have that $\displaystyle\sum_{p\le x,~p\in P_k}1/p<(\alpha_k+\eps/kd)\log \log x$. It follows that $\displaystyle \sum_{\substack{p\le x\\p\in P_k}}\log\left(1-\frac{1}{p}\right)=-\sum_{\substack{p\le x\\p\in P_k}}\frac{1}{p}+O(1)\ge -(\alpha_k+\eps/kd)\log\log x+O(1)$, so $\displaystyle \prod_{\substack{p\le x\\p\in P_k}}\left(1-\frac{1}{p}\right)\ge c_k\cdot (\log x)^{-\alpha_k-\eps/kd}$, where $c_k$ is a constant that does not depend on $\eps$. Consequently, 
$$\Pi_2'\sim C\cdot (\log x)^{-\alpha_1-2\alpha_2-\dots-d\alpha_d-\eps},$$
where $C$ is a constant independent of $\eps$, so we are left with taking one more careful look at the quantity $\abs{G}(\alpha_1+2\alpha_2+\dots+d\alpha_d)=\beta_1+2\beta_2+\dots+d\beta_d$. Now, this sum counts the number of pairs $(\sigma,y)$, where $\sigma$ is an element of the Galois group $G$ and $y$ is a fixed point of $\sigma$, by first considering permutations of $G$. If we consider first the roots $y$ of $P$, we see that we need to count how many automorphisms of $G$ fix $y$. As established in \ref{isotropy}, this number is the same for all roots. Therefore, the exponent we just obtained can be put in the shape $-qd-\eps$. 

Finally, putting $\Pi_1$, $\Pi_2$ and $\Pi_3$ together, we obtain that $\varphi_P(n)$ is bounded below by $cn\cdot (\log \log n)^{-qd-\eps}$, for all $n$ that are greater than some $n(\eps)$. Moreover, we see that the constant $\displaystyle c=C\cdot 0.99\cdot \prod_{p\le d}\frac{1}{p}$ depends on $d=\deg P$ only, and it has nothing to do with $\eps$.
\end{proof}

\section{Final Remarks}

The extra $-\eps$ that appears at the exponent of the lower bound comes from the fact that we do not know how to express the sum $\displaystyle \sum_{\substack{p\in P_k\\p\le x}}\frac{1}{p}$ clean in terms of $\log \log x$. Dirichlet density (which is implied by natural density) only goes this far. 

A way to avoid $-\eps$ is by directly using the bound $f(p)\le d$ for all large primes $p$. This way, the final lower bound would even have an explicit constant, however the exponent would only be $-d$, which is weaker than the result here. 

A final question might arise, and that is: what if $q=1$? In that case, the lower bound here would be even smaller than the safer one, with $-d$ as an exponent. 

If $q=1$ i.e. all automorphisms of $G$ fix all roots of $f$, then $G$ consists only of the identity automorphism, meaning $K=\QQ$. In that case, as $P$ is irreducible, we would obtain that $P$ is of degree $1$. However, this situation is very similar to the analysis of $\varphi(n)$ and consequently gives the same lower bound as the one for $\varphi(n)$, since all the trouble regarding the degree $d$ of $P$ is gone. 

So the result presented here is especially relevant for $\deg P\ge 2$ as, again, $\deg P=1$ could instead be put in the same situation as the classic $\varphi$.

\end{document}